\documentclass{article}
\usepackage{amssymb,amsmath,amsfonts,amsthm}
\usepackage{latexsym}
\usepackage{mathrsfs}
\usepackage{graphics}
\usepackage{tikz}
\usetikzlibrary{shapes,arrows,calc}
\usepackage{indentfirst}

\usepackage{hyperref}
\usepackage{enumerate}
\usepackage{mathtools}
\usepackage{comment}
\usepackage{xparse}
\usepackage{bm}
\usepackage{subcaption}
\usepackage{caption}
\usepackage{graphicx} 
\usetikzlibrary{positioning} 
\hypersetup{colorlinks = true, linkcolor = blue, citecolor = blue, urlcolor = blue}
\usepackage{float}
\usepackage[section]{placeins}
\usepackage{lineno}

\allowdisplaybreaks

\newtheorem{innercustomthm}{Theorem}
\newenvironment{customthm}[1]
  {\renewcommand\theinnercustomthm{#1}\innercustomthm}
  {\endinnercustomthm}

\newtheorem*{thm*}{Theorem}
\newtheorem{thm}{Theorem}
\newtheorem{lem}[thm]{Lemma}
\newtheorem{pro}[thm]{Proposition}

\newtheorem{cor}[thm]{Corollary}
\newtheorem{conj}[thm]{Conjecture}

\newtheorem{ques}[thm]{Question}

\newcommand{\N}{\mathbb{N}}
\newcommand{\Z}{\mathbb{Z}}

\newcommand{\col}{\mathrm{col}}

\usepackage[margin=1in]{geometry}
\usepackage{amsmath,amssymb}
\usepackage{amsthm}

\newtheorem*{claim*}{Claim}

\usepackage{xparse} 
\usepackage{etoolbox}
\NewDocumentCommand{\angles}{mmO{1}}{%
  \ensuremath{%
    \ifcase\numexpr#3\relax
      \PackageError{angles}{Nesting level must be >= 1}{You passed #3}%
    \or
      #1\langle #2 \rangle
    \or
      #1\langle\langle #2 \rangle\rangle
    \or
      #1\langle\langle\langle #2 \rangle\rangle\rangle
    \else
      \PackageError{angles}{Unsupported nesting level #3}{Only nesting levels 1--3 are supported}
    \fi
  }%
}
\usetikzlibrary{fit,matrix,positioning,backgrounds,shapes.geometric}
\newsavebox{\KTwoFourFigureBox}

\begin{document}
\title{Improved Bounds on Proportional Choice Numbers}

\author{Jeffrey A. Mudrock \thanks{Department of Mathematics and Statistics, University of South Alabama, Mobile, AL, USA (mudrock@southalabama.edu)}
}
\maketitle
\begin{abstract}

Proportional choosability is a list analogue of equitable coloring introduced in 2019 that imposes both upper and lower bounds on color class sizes based on the availability of each color. The proportional choice number of a graph $G$, denoted by $\chi_{pc}(G)$, is the smallest positive integer $k$ such that $G$ is proportionally $k$-choosable. We consider two previously posed questions regarding proportional choice numbers: (1) Is there a constant $C$ such that $\chi_{pc}(P_n) \leq C$ for every $n \in \N$? and (2) Is it the case that for any graph $G$, $\chi_{pc}(G) \leq \Delta(G)+1$ where $\Delta(G)$ is the maximum degree of $G$? In this paper we answer both questions: the first in the affirmative and the second in the negative.  Specifically, we prove that $\chi_{pc}(G) \leq 2^{\chi'(G)}$ for every graph $G$ where $\chi'(G)$ is the chromatic index of $G$. Consequently, $\chi_{pc}(P_n) \leq 4$ for every $n \in \N$. We also prove that when $G$ is the disjoint union of $(a+b-1)$ copies of $K_{a,b}$, $\chi_{pc}(G)=a+b$. In particular, for each $m \in \N$, there is a graph $G$ with $\Delta(G)=m$ and $\chi_{pc}(G)=2\Delta(G)$.



\end{abstract}

\section{Introduction} \label{intro}

In this paper, all graphs are nonempty, finite, simple graphs.  Generally speaking, we follow West~\cite{W01} for terminology and notation.  The set of natural numbers is $\N = \{1,2,3, \ldots \}$.  For $m \in \N$, we write $[m]$ for the set $\{1, \ldots, m \}$, and we let $[0] = \emptyset$.  We write $K_{l,n}$ for complete bipartite graphs with partite sets of size $l$ and $n$.  When we write $G = q K_{l,n}$ we mean that $G$ is $q$ pairwise vertex disjoint copies of $K_{l,n}$.  For a graph $G$, we use $d_G(v)$ for the degree of a vertex $v$ in $G$, and we use $\Delta(G)$ for the maximum degree of the vertices in $G$.  When $u,v \in V(G)$ are adjacent in $G$, we write $uv$ or $vu$ for the edge connecting $u$ and $v$.  When $G_1$ and $G_2$ are vertex disjoint we write $G_1 + G_2$ or $\sum_{i=1}^2 G_i$ for their disjoint union.  

\subsection{Equitable List Coloring and Proportional Coloring}

The study of equitable coloring began with a conjecture of Erd\H{o}s in 1964~\cite{E64}, although the concept was not formally introduced until Meyer's work in 1973~\cite{M73}. An \emph{equitable $k$-coloring} of a graph $G$ is a proper $k$-coloring $f$ such that the sizes of any two color classes differ by at most one. Here, we adopt the convention that a $k$-coloring has exactly $k$ color classes, some of which may be empty. Consequently, each color class in an equitable $k$-coloring has size either $\lceil |V(G)|/k \rceil$ or $\lfloor |V(G)|/k \rfloor$. We say that a graph $G$ is \emph{equitably $k$-colorable} if it admits an equitable $k$-coloring. 

List coloring was introduced independently by Vizing~\cite{V76} and Erd\H{o}s, Rubin, and Taylor~\cite{ER80} in the 1970s. A \emph{list assignment} $L$ for a graph $G$ is a function that assigns to each vertex of $G$ a list of colors. If $|L(v)|=k$ for every $v \in V(G)$, then $L$ is called a \emph{$k$-assignment} for $G$. A graph $G$ is \emph{$L$-colorable} if there exists a proper coloring $f$ of $G$ such that $f(v) \in L(v)$ for every $v \in V(G)$. Such a coloring is called a \emph{proper $L$-coloring} of $G$. We say that $G$ is \emph{$k$-choosable} if it admits a proper $L$-coloring for every $k$-assignment $L$ of $G$. Both list coloring and equitable coloring have been extensively studied since their introduction.

In 2003, Kostochka, Pelsmajer, and West introduced a list analogue of equitable coloring, called \emph{equitable choosability}~\cite{KP03}, which has since received considerable attention (e.g.,~\cite{KM18, KM22, KK13, MC21, ZB15}). Suppose that $L$ is a $k$-assignment for a graph $G$. A proper $L$-coloring of $G$ is \emph{equitable} if each color appears on at most $\lceil |V(G)|/k \rceil$ vertices. Such a coloring is called an \emph{equitable $L$-coloring} of $G$, and we say that $G$ is \emph{equitably $L$-colorable} if it admits an equitable $L$-coloring. We say that $G$ is \emph{equitably $k$-choosable} if $G$ is equitably $L$-colorable for every $k$-assignment $L$ of $G$. Observe that equitable choosability imposes the same upper bound on color class sizes as equitable coloring, but it does not require the sizes of the color classes to be balanced.

In 2019, another list analogue of equitable coloring, called \emph{proportional choosability}~\cite{KM19}, was introduced. It imposes both upper and lower bounds on color class sizes based on the availability of each color. Since a color may appear in as few as one list, we can't require every color class to have size at least $\lfloor |V(G)|/k \rfloor$. Instead, proportional choosability uses bounds that depend on the number of lists containing each color.

Let $L$ be a list assignment for a graph $G$. The \emph{palette of colors associated with $L$} is the set $\bigcup_{v \in V(G)} L(v)$, which we denote by $\mathcal{L}$ throughout this paper. For each color $c \in \mathcal{L}$, the \emph{multiplicity of $c$ in $L$}, denoted by $\eta_L(c)$ (or simply $\eta(c)$ when $L$ is clear from context), is the number of lists in which $c$ appears. In particular, $\eta_L(c) = |\{v \in V(G) : c \in L(v)\}|$.

Suppose that $L$ is a $k$-assignment for a graph $G$. A proper $L$-coloring $f$ of $G$ is called a \emph{proportional $L$-coloring} if, for every $c \in \mathcal{L}$, the color class $f^{-1}(c)$ has size either $\lfloor \eta(c)/k \rfloor$ or $\lceil \eta(c)/k \rceil$. We say that $G$ is \emph{proportionally $L$-colorable} if it admits a proportional $L$-coloring.  Furthermore, we say that $G$ is \emph{proportionally $k$-choosable} if it is proportionally $L$-colorable for every $k$-assignment $L$ of $G$.

We now recall several basic properties of proportional $k$-choosability. Proofs and further discussion of these properties can be found in~\cite{KM19,M18a}.  For the first property, recall that if a graph is $k$-choosable, then it is $k$-colorable. However, the analogous implication fails for equitable choosability: a graph may be equitably $k$-choosable without being equitably $k$-colorable. More recently, Kierstead, Kostochka, and Xiang introduced \emph{strong equitable choosability}, which addresses this shortcoming~\cite{KK25}. Specifically, suppose $L$ is a $k$-assignment for a graph $G$, and write $|V(G)|=qk+r$, where $0\leq r<k$. An equitable $L$-coloring is \emph{strongly equitable} if at most $r$ color classes have size $q+1$. A graph $G$ is \emph{strongly equitably $k$-choosable} if it admits a strongly equitable $L$-coloring for every $k$-assignment $L$ for $G$. Strong equitable $k$-choosability implies both equitable $k$-choosability and equitable $k$-colorability. Proportional choosability also guarantees both properties as the following result shows.

\begin{pro}[\cite{KM19}] \label{pro: motivation}
If $G$ is proportionally $k$-choosable, then $G$ is both equitably $k$-choosable and equitably $k$-colorable.
\end{pro}

The next two properties highlight additional advantages of proportional choosability. In particular, proportional choosability is monotone in $k$ and is also a monotone graph property, meaning that it is preserved under taking subgraphs. While these properties hold for ordinary colorability and choosability, neither holds in general for equitable choosability or strong equitable choosability. For example, $K_{2,3,3}$ is strongly equitably $3$-choosable (using its $3$-choosability~\cite{GM98}) but not strongly equitably $4$-choosable, and $K_{1,3}\cup K_1$ is strongly equitably $2$-choosable even though its subgraph $K_{1,3}$ is not.

\begin{pro}[\cite{KM19}] \label{pro: monoink}
If $G$ is proportionally $k$-choosable, then $G$ is proportionally $(k+1)$-choosable.
\end{pro}

\begin{pro}[\cite{KM19}] \label{lem: monotone}
Suppose $H$ is a subgraph of $G$. If $G$ is proportionally $k$-choosable, then $H$ is proportionally $k$-choosable.
\end{pro}

For any graph $G$, the \emph{proportional choice number} of $G$, denoted by $\chi_{pc}(G)$, is the smallest positive integer $k$ such that $G$ is proportionally $k$-choosable. By Proposition~\ref{pro: monoink}, a graph $G$ is proportionally $k$-choosable if and only if $k \geq \chi_{pc}(G)$.

Although a complete characterization of equitably $2$-choosable graphs remains elusive, a characterization of proportionally $2$-choosable graphs is established in~\cite{KM20}. Recall that a \emph{linear forest} is a forest whose components are paths.

\begin{thm}[\cite{KM20}] \label{thm: full2character}
A graph $G$ is proportionally $2$-choosable if and only if $G$ is a linear forest in which the largest component has at most five vertices and every other component has at most two vertices.
\end{thm}

In particular, Theorem~\ref{thm: full2character} implies that $\chi_{pc}(P_n)=2$ for $2\leq n\leq 5$, while $\chi_{pc}(P_n)\geq 3$ for every $n\geq 6$. Despite this characterization, relatively little is known about the proportional choice numbers of longer paths. This leads naturally to the following two questions which provide one of the motivations for the present work.

\begin{ques}[\cite{KM20}] \label{ques: proportionalpaths1}
For each $n \geq 6$, what is the value of $\chi_{pc}(P_n)$?
\end{ques}

\begin{ques}[\cite{KM20}] \label{ques: proportionalpaths2}
Does there exist a constant $C$ such that $\chi_{pc}(P_n) \leq C$ for every $n \in \N$?
\end{ques}

It is worth mentioning that related results concerning the proportional choosability of paths when the palette of colors is restricted were obtained in~\cite{MP22}.  Another motivation for the present work is a general upper bound on proportional choice numbers. The following general upper bound is known, and the subsequent question was posed in the first paper on proportional choosability.

\begin{thm}[\cite{KM19}] \label{thm: smallorder}
Any graph $G$ with $\Delta(G) \geq 1$ satisfies
$$\chi_{pc}(G) \leq \Delta(G) + \left \lceil \frac{|V(G)|}{2} \right \rceil.$$
\end{thm}

\begin{ques}[\cite{KM19}] \label{ques: fullHS}
For any graph $G$, does it follow that $\chi_{pc}(G) \leq \Delta(G)+1$?
\end{ques}

\subsection{Outline and Summary of Results}
In this paper we answer Questions~\ref{ques: proportionalpaths2} and~\ref{ques: fullHS}.  We now present an outline.  In Section~\ref{generalreduction}, we begin by using ideas from matching theory and linear programming to prove the following result.

\begin{thm} \label{thm: match}
Suppose $G$ is a graph and $M$ is a matching in $G$. Then, $\chi_{pc}(G) \leq 2\chi_{pc}(G-M)$.
\end{thm}

The \emph{chromatic index} of a graph $G$, denoted by $\chi'(G)$, is the smallest nonnegative integer $m$ such that $G$ has a proper $m$-edge coloring.  Theorem~\ref{thm: match} allows us to easily prove the following general upper bound on the proportional choice number of a graph from which a positive answer to Question~\ref{ques: proportionalpaths2} immediately follows.

\begin{thm} \label{thm: chromindex}
For any graph $G$, $\chi_{pc}(G) \leq 2^{\chi'(G)}$.
\end{thm}

Since the chromatic index of any path is at most 2, Theorem~\ref{thm: chromindex} implies that $\chi_{pc}(P_n) \leq 4$ for every $n \in \N$.  The bound of 4 actually applies to any bipartite graph with maximum degree 2.  Importantly, Question~\ref{ques: proportionalpaths1} is still open.  In fact, we conjecture that our bound is never tight for paths.

\begin{conj} \label{conj: its3}
For any $n \geq 6$, $\chi_{pc}(P_n) = 3$.
\end{conj}

We end Section~\ref{generalreduction} by showing how Theorems~\ref{thm: full2character} and~\ref{thm: match} can be used together to show that the proportional choice number of any odd cycle is at most 4 even though the chromatic index of an odd cycle is 3.

In Section~\ref{tightness} we consider the tightness of the bound in Theorem~\ref{thm: chromindex}, and in the process, show that the answer to Question~\ref{ques: fullHS} is no.  In particular, we prove the following.

\begin{thm}\label{thm: Kab}
If $a,b \in \N$ satisfy $1 \leq a \leq b$ and $G=(a+b-1)K_{a,b}$, then $\chi_{pc}(G)=a+b = \Delta(G)+a$.
\end{thm}

The bound in Theorem~\ref{thm: chromindex} is clearly tight when $\chi'(G)=1$ by Theorem~\ref{thm: full2character} since any graph with chromatic index 1 consists of one or more copies of $K_2$ possibly together with isolated vertices. Theorem~\ref{thm: Kab} gives an interesting example of tightness.  In particular, when $G = 3 K_{2,2} = 3 C_4$ its chromatic index is 2, and its proportional choice number is 4.

Notice that Theorem~\ref{thm: Kab} also tells us that when $G = (2d-1)K_{d,d}$ for some $d \in \N$, $\chi_{pc}(G) = 2d = 2 \Delta(G)$ which shows that the answer to Question~\ref{ques: fullHS} is no in a rather strong sense.  With Theorem~\ref{thm: Kab} in mind, we conjecture another general upper bound on the proportional choice number of a graph.  Recall that the \emph{coloring number} of a graph $G$, denoted by $\col(G)$, is the smallest $d \in \N$ such that there is an ordering of the vertices in $V(G)$ such that each vertex in the ordering has at most $d-1$ neighbors preceding it in the ordering.  We conjecture the following.

\begin{conj} \label{conj: upper}
For every graph $G$ with $\Delta(G) \geq 1$, $\chi_{pc}(G) \leq \Delta(G) + \col(G) - 1$.
\end{conj}  

Note that since $\col(K_{a,b}) = a+1$ when $a \leq b$, Theorem~\ref{thm: Kab} would demonstrate the tightness of Conjecture~\ref{conj: upper} if true.  Additionally, if Conjecture~\ref{conj: upper} is true it would immediately imply Conjecture~\ref{conj: its3} since $\col(P_n)=2$ whenever $n \geq 2$.

\section{A Matching Reduction and Consequences}\label{generalreduction}

We begin this section by proving Theorem~\ref{thm: match}.

\begin{proof}
The result is trivially true when $M$ is empty. So, we may suppose $M \neq \emptyset$. Suppose that $\chi_{pc}(G-M) = t$, and $L$ is an arbitrary $2t$-assignment of $G$. We must show that $G$ is proportionally $L$-colorable. Throughout the proof let $\mathcal{L} = \bigcup_{v \in V(G)} L(v)$.

Now, for each $c \in \mathcal{L}$, let $O(c) = \{(c,v) \in \{c\} \times V(G) : c \in L(v)\}$. Also, let $\mathcal{B}(c)$ be an arbitrarily chosen partition of $O(c)$ into blocks of size 2 and at most one block of size 1 with the property that $\{(c,u),(c,v)\} \in \mathcal{B}(c)$ whenever $uv \in M$ and $(c,u),(c,v) \in O(c)$. Since $M$ is a matching, such a partition exists. Also, $\mathcal{B}(c)$ will contain a block of size 1 if and only if $\eta_L(c)$ is odd.

Next, let $\mathcal{B} = \bigcup_{c \in \mathcal{L}} \mathcal{B}(c)$. Let $B$ be the bipartite graph with bipartition $V(G)$, $\mathcal{B}$, and edges such that each $v \in V(G)$ is adjacent to a set in $\mathcal{B}$ if and only if the set contains an ordered pair with $v$ as the second coordinate. Note that $d_B(v) = 2t$ for each $v \in V(G)$, and for each $X \in \mathcal{B}$, $d_B(X) = |X|$. Now, we claim there is a spanning subgraph $S$ of $B$ such that each $v \in V(G)$ has degree $t$ in $S$, each $X \in \mathcal{B}$ with $|X|=2$ has degree 1 in $S$, and all other elements of $\mathcal{B}$ have degree at most 1 in $S$.

To see such a subgraph exists, consider the following feasibility LP. Associate with each edge $e \in E(B)$ a variable $x_e$. The LP consists of the following constraints:
\begin{align*}
\sum_{e \ni v} x_e &= t \quad (v \in V(G)), & \sum_{e \ni X} x_e &= 1 \quad (X \in \mathcal{B}, |X|=2),\\
\sum_{e \ni X} x_e &\leq 1 \quad (X \in \mathcal{B}, |X|=1), \text{ and } & 0 \leq x_e &\leq 1 \quad (e \in E(B)).
\end{align*}
Observe that setting $x_e=1/2$ for each $e \in E(B)$ gives a feasible solution.  Moreover, since $B$ is bipartite, its vertex-edge incidence matrix is totally unimodular. Replacing equality constraints by pairs of inequalities and adjoining the bounds $0 \leq x_e \leq 1$ preserves total unimodularity. Consequently, since all the right-hand sides of the constraints are integers, the feasible polytope is integral by a theorem of Hoffman and Kruskal (see e.g.,~\cite[Theorem 19.1]{S86}). Thus, since the polytope is nonempty and bounded, there is a feasible solution in which $x_e \in \{0,1\}$ for each $e \in E(B)$.  The edges corresponding to variables with value 1 form the edge set of the desired spanning subgraph $S$ of $B$.

Now, for each $v \in V(G)$, let \[L'(v) = \{c \in \mathcal{L} : \text{$v$ is adjacent in $S$ to a set containing an element with first coordinate $c$} \}.\] Notice $L'(v) \subseteq L(v)$ for each $v \in V(G)$ and $L'$ is a $t$-assignment for $G$.

Suppose that for $c \in \mathcal{L}$, $\eta_L(c)=2tq+r$ where $q,r \in \Z$ and $0 \leq r \leq 2t-1$. By the degree properties of the vertices in $S$, we have
\[
\eta_{L'}(c) \in \left\{tq+\left\lfloor \frac{r}{2}\right\rfloor,tq+\left\lceil \frac{r}{2}\right\rceil\right\}.
\]
Notice that when $r=0$, $\eta_{L'}(c)=tq$, and when $r>0$, $tq \leq \eta_{L'}(c) \leq t(q+1)$. Consequently, a proportional $L'$-coloring uses $c$ exactly $q$ times when $r=0$, and either $q$ or $q+1$ times when $r>0$. Thus, a proportional $L'$-coloring of $G$ is also a proportional $L$-coloring of $G$.

Since $L'$ is also a $t$-assignment for $G-M$, there is a proportional $L'$-coloring of $G-M$ which we will call $f$. We claim that $f$ is also a proportional $L'$-coloring of $G$. Since $f$ uses colors an appropriate number of times for a proportional coloring, we need only show that $f$ is a proper $L'$-coloring of $G$. So, for the sake of contradiction suppose that $uv \in E(G)$ and $f(u) = z = f(v)$. Since $f$ is a proper $L'$-coloring of $G-M$, $uv \in M$. However, this would mean that $\{(z,u),(z,v)\} \in \mathcal{B}$, and since this set has degree 1 in $S$, $z \notin L'(u) \cap L'(v)$ which is a contradiction.

So, $f$ is indeed a proportional $L'$-coloring of $G$ and consequently it is also a proportional $L$-coloring of $G$.
\end{proof}

Now, recall that the \emph{chromatic index} of a graph $G$, denoted by $\chi'(G)$, is the smallest nonnegative integer $m$ such that $G$ has a proper $m$-edge coloring; that is, a coloring that colors the edges of $G$ with colors from $[m]$ such that any two edges sharing an endpoint receive different colors.  Note that when $G$ is edgeless, the chromatic index of $G$ is zero.  By a famous result of Vizing~\cite{V64}, when $G$ is simple, $\chi'(G) \in \{\Delta(G),\Delta(G)+1\}$, and by a famous result of K\"onig~\cite{K16}, $\chi'(G)=\Delta(G)$ when $G$ is bipartite.
  We are now ready to prove Theorem~\ref{thm: chromindex} which we restate.

\begin{customthm}{\ref{thm: chromindex}}
For any graph $G$, $\chi_{pc}(G) \leq 2^{\chi'(G)}$.
\end{customthm}

\begin{proof}
The result is clear when $G$ is edgeless.  So, we may suppose $E(G) \neq \emptyset$.  Let $m = \chi'(G)$, and suppose $f:E(G) \rightarrow [m]$ is a proper $m$-edge coloring of $G$.  For each $i \in [m]$, let $M_i = f^{-1}(i)$.  Note that $\{M_1, \dots, M_m \}$ is a partition of $E(G)$ and since $f$ is proper, each element of this set is a matching in $G$.  Let $G_0 = G$, and for each $i \in [m]$, let $G_i = G - \bigcup_{j=1}^i M_j$.  Then, Theorem~\ref{thm: match} tells us for each $i \in [m]$, $\chi_{pc}(G_{i-1}) \leq 2\chi_{pc}(G_{i})$.  Since $\chi_{pc}(G_m) = 1$, it follows that $\chi_{pc}(G_i) \leq 2^{m-i}$ for each $i \in (\{0\} \cup [m])$.  Consequently, $\chi_{pc}(G) = \chi_{pc}(G_0) \leq 2^m$. 
\end{proof}

By Vizing and K\"onig's aforementioned results, the following corollary is immediate.

\begin{cor} \label{cor: answers}
The following statements hold.\\
(i)  If $G$ is a simple graph, $\chi_{pc}(G) \leq 2^{\Delta(G)+1}$. \\
(ii) If $G$ is a bipartite graph, then $\chi_{pc}(G) \leq 2^{\Delta(G)}$.  Consequently, $\chi_{pc}(P_n) \leq 4$ and $\chi_{pc}(C_{2n+2}) \leq 4$ for each $n \in \N$.
\end{cor}

We also easily get the following for odd cycles.

\begin{cor} \label{cor: oddcycle}
For any $n \in \N$, $\chi_{pc}(C_{2n+1}) \leq 4$.
\end{cor}

\begin{proof}
Suppose $G = C_{2n+1}$ and the vertices of $G$ in cyclic order are $v_1, v_2, \ldots, v_{2n+1}$.  Note that $M = \{v_{2i-1}v_{2i} : i \in [n] \}$ is a matching in $G$.  Additionally, $G - M$ consists of a $P_3$ and $(n-1)$ copies of $K_2$.  So, Theorem~\ref{thm: full2character} implies that $\chi_{pc}(G-M)=2$, and the result follows by Theorem~\ref{thm: match}. 
\end{proof}

\section{Tightness and Answering Question~\ref{ques: fullHS}} \label{tightness}
It is natural to wonder whether the bounds in Theorems~\ref{thm: match} and \ref{thm: chromindex} are tight.  In this section we explore this.  We begin with a lemma that will be useful for establishing lower bounds on certain proportional choice numbers.

\begin{lem} \label{lem: obstruct}
Suppose $G$ is a graph and for some $m \geq 2$ and $k \geq 2$, $\{A_1, \dots, A_m\}$ is a partition of $V(G)$ into blocks of size $k$.  If there is no independent set $I$ in $G$ with the property $|I \cap A_i| = 1$ for each $i \in [m]$, then $\chi_{pc}(G) > k$.
\end{lem}

\begin{proof}
Suppose that $\mathcal{P} = \{P_1, \ldots, P_m\}$ is an arbitrarily chosen partition of $\{2,3,\dots,m(k-1)+1\}$ into blocks of size $k-1$.  Now, let $L$ be the $k$-assignment for $G$ such that for each $i \in [m]$ and $v \in A_i$, $L(v) = \{1\} \cup P_i$.  Notice that $\eta_L(1) = mk$, and the multiplicity of every color in $\{2,3,\dots,m(k-1)+1\}$ is $k$.

Now, for the sake of contradiction, suppose that $f$ is a proportional $L$-coloring of $G$.  Notice that for each $i \in [m]$ and $c \in P_i$, since the color $c$ only appears in the lists corresponding to vertices in $A_i$, $f$ must color exactly one vertex in $A_i$ with $c$.  It follows that for each $i \in [m]$, $f$ colors exactly one vertex in $A_i$ with 1.  However, this implies that $f^{-1}(1)$ is an independent set in $G$ with the property $|f^{-1}(1) \cap A_i| = 1$ for each $i \in [m]$ which is a contradiction.  Consequently, no proportional $L$-coloring of $G$ exists and $\chi_{pc}(G) > k$.
\end{proof}

It is important to note at this point that for any given graph $G$ with $\Delta(G) \geq 2$, the largest $k$ for which we could hope to apply Lemma~\ref{lem: obstruct} is $k = 2 \Delta(G) - 1$.   Indeed, a result of Haxell~\cite{H01} implies that if $G$ is a graph and $\{A_1,\ldots,A_m\}$ is a partition of $V(G)$ with $|A_i| \geq 2\Delta(G)$ for each $i \in [m]$, then there is an independent set $I$ in $G$ such that $|I \cap A_i|=1$ for each $i \in [m]$. 

Before we begin proving Theorem~\ref{thm: Kab}, we mention one more result that we need for our proof.

\begin{lem}[\cite{KM19}] \label{lem: fullcomponents}
If $G$ is a graph such that its largest component has $k$ vertices, then $\chi_{pc}(G) \leq k$.
\end{lem}

\begin{customthm}{\ref{thm: Kab}}
If $a,b \in \N$ satisfy $1 \leq a \leq b$ and $G=(a+b-1)K_{a,b}$, then $\chi_{pc}(G)=a+b = \Delta(G)+a$.
\end{customthm}

\begin{proof}
The result is clear when $a=b=1$. So, we may suppose that $a+b \geq 3$. Since each component of $G$ has $a+b$ vertices, we have $\chi_{pc}(G) \leq a+b$ by Lemma~\ref{lem: fullcomponents}. Now, we will use Lemma~\ref{lem: obstruct} to show that $\chi_{pc}(G)>a+b-1$.

Specifically, we will construct a partition of $V(G)$ into blocks of size $a+b-1$ satisfying the hypotheses of Lemma~\ref{lem: obstruct}. 
Our construction is an adaptation of Construction~3.3 of Szabó and Tardos~\cite{ST06}. The construction is inductive. Suppose that the components of $G$ are $G_0,\ldots,G_{a+b-2}$. For each such component $G_j$, suppose that $A_j,B_j$ is the bipartition of $G_j$ where $|A_j|=a$ and $|B_j|=b$.

First, note that there is no independent set in $G_0$ containing exactly one vertex from each of $A_0$ and $B_0$. Let $X_0=\{A_0\}$ and $Y_0=\{B_0\}$. Now proceed inductively for each $j \in [a-1]$ as follows. Suppose that $X_{j-1}$ and $Y_{j-1}$ have been constructed so that $X_{j-1} \cup Y_{j-1}$ is a partition of $V(\bigcup_{i=0}^{2j-2} G_i)$ such that no independent set in $G$ meets each block in exactly one vertex, and such that $X_{j-1}$ has exactly one block of size $a+j-1$, $Y_{j-1}$ has exactly one block of size $b+j-1$, and all other blocks have size $a+b-1$.

Let $Z$ be the unique block in $X_{j-1}$ of size $a+j-1$. Arbitrarily partition $Z$ into two sets $Z_1$ and $Z_2$ such that $|Z_1|=j$ and $|Z_2|=a-1$. Now, let
\[
X_j=(X_{j-1}-\{Z\})\cup\{Z_1\cup A_{2j-1},Z_2\cup B_{2j-1}\}.
\]
Similarly, let $W$ be the unique block in $Y_{j-1}$ of size $b+j-1$, and arbitrarily partition $W$ into two sets $W_1$ and $W_2$ such that $|W_1|=j$ and $|W_2|=b-1$. Let
\[
Y_j=(Y_{j-1}-\{W\})\cup\{W_1\cup B_{2j},W_2\cup A_{2j}\}.
\]
Notice that $X_j$ contains $j$ blocks of size $a+b-1$ and another block of size $a+j$. Also, $Y_j$ contains $j$ blocks of size $a+b-1$ and another block of size $b+j$.

We claim that $X_j \cup Y_j$ is a partition of $V(\bigcup_{i=0}^{2j} G_i)$ such that no independent set in $G$ meets each block in exactly one vertex. To see this, suppose for the sake of contradiction that $I$ is such an independent set. Since every vertex in $A_{2j-1}$ is adjacent to every vertex in $B_{2j-1}$, either the vertex in $I$ that is in $Z_1\cup A_{2j-1}$ or the vertex in $I$ that is in $Z_2\cup B_{2j-1}$ must belong to $Z$. Similarly, either the vertex in $I$ that is in $W_1\cup B_{2j}$ or the vertex in $I$ that is in $W_2\cup A_{2j}$ must belong to $W$. Keeping one such vertex from each of $Z$ and $W$, together with the vertices selected from all unchanged blocks, gives an independent set of $G$ with exactly one vertex in each block of $X_{j-1}\cup Y_{j-1}$ which is a contradiction.

Now, notice that every block in $Y_{a-1}$ has size $a+b-1$. Also, $X_{a-1}$ contains $a-1$ blocks of size $a+b-1$ and another block of size $2a-1$. If $a=b$, our construction is complete. Otherwise, we continue inductively for each $j \in \{a,\ldots,b-1\}$ as follows. Let $Z$ be the unique block in $X_{j-1}$ of size $a+j-1$, and arbitrarily partition $Z$ into two sets $Z_1$ and $Z_2$ such that $|Z_1|=j$ and $|Z_2|=a-1$. Let
\[
X_j=(X_{j-1}-\{Z\})\cup\{Z_1\cup A_{a+j-1},Z_2\cup B_{a+j-1}\}
\]
and let $Y_j=Y_{j-1}$. By the same argument used above for the construction of $X_j$, there is no independent set of $G$ with exactly one vertex in each block of $X_j\cup Y_j$. Moreover, $X_j$ contains $j$ blocks of size $a+b-1$ and another block of size $a+j$.

Consequently, $X_{b-1}\cup Y_{b-1}$ is a partition of $V(G)$ into blocks of size $a+b-1$ satisfying the hypotheses of Lemma~\ref{lem: obstruct}. By Lemma~\ref{lem: obstruct}, $\chi_{pc}(G)>a+b-1$, and the result follows.
\end{proof}

\begin{figure}[!htbp]
\centering
\begin{tikzpicture}[vertex/.style={circle,draw,fill=white,inner sep=0pt,minimum size=6pt},every node/.style={font=\small}]

\begin{scope}[xshift=0cm]
\draw (0,1.5)--(1.7,1.5);
\draw (0,1.5)--(1.7,0);
\draw (0,0)--(1.7,1.5);
\draw (0,0)--(1.7,0);
\node[vertex] (u1) at (0,1.5) {};
\node[vertex] (u3) at (0,0) {};
\node[vertex] (u2) at (1.7,1.5) {};
\node[vertex] (u4) at (1.7,0) {};
\node[above=4pt,align=center] at (u1) {$\{1,2,3\}$\\$u_1$};
\node[above=4pt,align=center] at (u2) {$\{1,6,7\}$\\$u_2$};
\node[below=4pt,align=center] at (u3) {$u_3$\\$\{1,4,5\}$};
\node[below=4pt,align=center] at (u4) {$u_4$\\$\{1,8,9\}$};
\end{scope}

\begin{scope}[xshift=4.5cm]
\draw (0,1.5)--(1.7,1.5);
\draw (0,1.5)--(1.7,0);
\draw (0,0)--(1.7,1.5);
\draw (0,0)--(1.7,0);
\node[vertex] (v1) at (0,1.5) {};
\node[vertex] (v3) at (0,0) {};
\node[vertex] (v2) at (1.7,1.5) {};
\node[vertex] (v4) at (1.7,0) {};
\node[above=4pt,align=center] at (v1) {$\{1,2,3\}$\\$v_1$};
\node[above=4pt,align=center] at (v2) {$\{1,4,5\}$\\$v_2$};
\node[below=4pt,align=center] at (v3) {$v_3$\\$\{1,2,3\}$};
\node[below=4pt,align=center] at (v4) {$v_4$\\$\{1,4,5\}$};
\end{scope}

\begin{scope}[xshift=9cm]
\draw (0,1.5)--(1.7,1.5);
\draw (0,1.5)--(1.7,0);
\draw (0,0)--(1.7,1.5);
\draw (0,0)--(1.7,0);
\node[vertex] (w1) at (0,1.5) {};
\node[vertex] (w3) at (0,0) {};
\node[vertex] (w2) at (1.7,1.5) {};
\node[vertex] (w4) at (1.7,0) {};
\node[above=4pt,align=center] at (w1) {$\{1,8,9\}$\\$w_1$};
\node[above=4pt,align=center] at (w2) {$\{1,6,7\}$\\$w_2$};
\node[below=4pt,align=center] at (w3) {$w_3$\\$\{1,8,9\}$};
\node[below=4pt,align=center] at (w4) {$w_4$\\$\{1,6,7\}$};
\end{scope}

\end{tikzpicture}

\caption{For $G=3C_4$ pictured above, a partition of $V(G)$ into four blocks of size three obtained from the construction in the proof of Theorem~\ref{thm: Kab} is $\{\{u_1,v_1,v_3\},\{u_3,v_2,v_4\},\{u_2,w_2,w_4\},\{u_4,w_1,w_3\}\}$.  The corresponding $3$-assignment $L$ is constructed as in the proof of Lemma~\ref{lem: obstruct}, and there is no proportional $L$-coloring of $G$.}

\label{fig:3C4lists}

\end{figure}
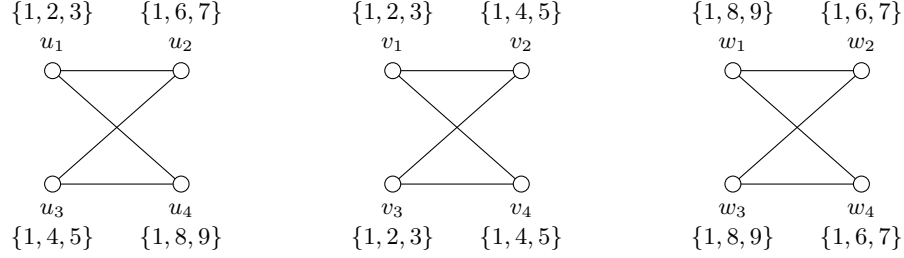

The result of Theorem~\ref{thm: Kab} for $a=1$ was already known (it follows from~\cite[Proposition 27]{KM19} and Lemma~\ref{lem: fullcomponents}).  Also, in~\cite{MH21} it is shown that $\chi_{pc}(K_{a,b})\leq b+1$ whenever $2\leq a\leq b$ and $a\leq\lfloor b/3\rfloor+2$, whereas Theorem~\ref{thm: Kab} gives $\chi_{pc}((a+b-1)K_{a,b})=a+b>b+1$ for these same values of $a$ and $b$.

For any $d\geq 2$, let $G=(2d-1)K_{d,d}$. By Theorem~\ref{thm: Kab}, we have $\chi_{pc}(G)=2d=2\Delta(G)$, which shows that the proposed upper bound of $\Delta(G)+1$ in Question~\ref{ques: fullHS} fails for every maximum degree above 1. Moreover, when $G=3C_4$, we have $\chi'(G)=2$ and $\chi_{pc}(G)=4=2^{\chi'(G)}$ which means the bound in Theorem~\ref{thm: chromindex} is tight when the chromatic index is 2. Figure~\ref{fig:3C4lists} illustrates a $3$-assignment $L$ of $G=3C_4$ arising from our construction above such that $G$ has no proportional $L$-coloring. Moreover, when $a=b=d$, our construction produces a partition into blocks of size $2d-1=2\Delta(G)-1$ such that no independent set meets each block in exactly one vertex. Recall that Haxell's result~\cite{H01} guarantees such an independent set whenever the blocks have size at least $2\Delta(G)$.

There are also counterexamples for Question~\ref{ques: fullHS} that are connected. For $d\geq 3$, let $H$ be obtained from a copy of $(2d-1)K_{d,d}$ by arbitrarily ordering the components and adding an edge between each pair of consecutive components, choosing the edges so that no vertex is incident with more than one added edge. Then, $H$ is connected with $\Delta(H)=d+1$. By Proposition~\ref{lem: monotone} and Theorem~\ref{thm: Kab}, we have $\chi_{pc}(H)\geq 2d>d+2=\Delta(H)+1$.

Finally, notice that if $G=(a+b-1)K_{a,b}$, where $1\leq a\leq b$, then $\Delta(G)=b$ and $\col(G)=a+1$. Consequently, Theorem~\ref{thm: Kab} gives $\chi_{pc}(G)=\Delta(G)+\col(G)-1$. This motivates Conjecture~\ref{conj: upper} which states that $\chi_{pc}(G)\leq\Delta(G)+\col(G)-1$ for every graph $G$ with $\Delta(G)\geq1$. If true, this conjecture would imply that $\chi_{pc}(P_n)=3$ for every $n\geq6$, answering Question~\ref{ques: proportionalpaths1}.

\par\medskip
\noindent\textbf{Acknowledgment.} The author is grateful to Alexandr Kostochka who posed Question~\ref{ques: proportionalpaths2} to the author in February of 2018.

\par\medskip
\noindent\textbf{Declaration of generative AI and AI-assisted technologies in the manuscript preparation process.} During the preparation of this manuscript, the author used ChatGPT (OpenAI) and Claude (Anthropic) to discuss and examine mathematical arguments, explore possible connections, assist with the editing and presentation of the manuscript, generate TikZ code for a figure, and check citations. The author reviewed and revised AI-assisted material and takes full responsibility for the final content of the manuscript.

\bibliographystyle{hplain}
\bibliography{bibliography}{}

\end{document}